\documentclass[11pt, a4paper]{amsart}
\usepackage{fullpage}
\usepackage{amssymb}
\usepackage{commath}
\usepackage{empheq}
\usepackage{amsmath, amsthm}
\usepackage[foot]{amsaddr}

\usepackage{tikz-cd}

\usepackage{url}

\newcommand{\p}{\partial}
\newcommand{\mr}{\mathrm}
\newcommand{\SB}{B_A^{sym}}
\newcommand{\lb}{\left[}
\newcommand{\rb}{\right]}
\newcommand{\ol}{\overline}

\newcommand{\xra}{\xrightarrow}
\newcommand{\DS}{\Delta S}
\newcommand{\ve}{\varepsilon}
\renewcommand{\phi}{\varphi}
\newcommand{\HC}{\Delta H}
\newcommand{\DHP}{\Delta H_{+}}
\newcommand{\HB}{\mathsf{H}_A}
\newcommand{\HBI}{\mathsf{H}_I}
\newcommand{\HBP}{\mathsf{H}_{A+}}
\newcommand{\HBM}{\mathsf{H}_M}
\newcommand{\EDH}{\mathrm{Epi}\Delta H}
\newcommand{\cyc}{C_2}
\newcommand{\ifas}{\mathcal{IF}(as)}
\newcommand{\llb}{\left\lbrace}
\newcommand{\rrb}{\right\rbrace}
\newcommand{\llv}{\left\lvert}
\newcommand{\rrv}{\right\rvert}
\newcommand{\llangle}{\left\langle}
\newcommand{\rrangle}{\right\rangle}
\renewcommand{\H}{\mathcal{H}}
\newcommand{\RNum}[1]{\MakeUppercase{\romannumeral #1}}
\newcommand{\CM}{\mathbf{CMod}}
\newcommand{\MC}{\mathbf{ModC}}
\newcommand{\kmod}{\mathbf{Mod}_k}
\newcommand{\C}{\mathbf{C}}
\newcommand{\hocolim}{\underset{\DHP}{\mr{hocolim}}}

\mathchardef\mhyphen="2D

\newtheorem{thm}{Theorem}[section]

\theoremstyle{plain}
\newtheorem{prop}[thm]{Proposition}
\newtheorem{lem}[thm]{Lemma}
\newtheorem{cor}[thm]{Corollary}

\theoremstyle{definition}
\newtheorem{defn}[thm]{Definition}

\theoremstyle{remark}
\newtheorem{rem}[thm]{Remark}
\newtheorem{eg}[thm]{Example}

\makeatletter
\@namedef{subjclassname@1991}{\textup{2020} Mathematics Subject Classification}
\makeatother

\title{Hyperoctahedral Homology for Involutive Algebras}
\author{Daniel Graves}
\address{School of Mathematics and Statistics, University of Sheffield, Sheffield, S3 7RH, UK}
\date{}

\begin{document}


\maketitle

\begin{abstract}
Hyperoctahedral homology is the homology theory associated to the hyperoctahedral crossed simplicial group. It is defined for involutive algebras over a commutative ring using functor homology and the hyperoctahedral bar construction of Fiedorowicz. The main result of the paper proves that hyperoctahedral homology is related to equivariant stable homotopy theory: for a discrete group of odd order, the hyperoctahedral homology of the group algebra is isomorphic to the homology of the fixed points under the involution of an equivariant infinite loop space built from the classifying space of the group.
\end{abstract}

\section*{Introduction}
\emph{Hyperoctahedral homology for involutive algebras} was introduced by Fiedorowicz \cite[Section 2]{Fie}. It is the homology theory associated to the \emph{hyperoctahedral crossed simplicial group} \cite[Section 3]{FL}. Fiedorowicz and Loday \cite[6.16]{FL} had shown that the homology theory constructed from the hyperoctahedral crossed simplicial group via a contravariant bar construction, analogously to cyclic homology, was isomorphic to Hochschild homology and therefore did not detect the action of the hyperoctahedral groups. Fiedorowicz demonstrated that a covariant bar construction did detect this action and sketched results connecting the hyperoctahedral homology of monoid algebras and group algebras to May's two-sided bar construction and infinite loop spaces, though these were never published.

In Section \ref{hyp-cat-sec} we recall the hyperoctahedral groups. We recall the \emph{hyperoctahedral category} $\HC$ associated to the hyperoctahedral crossed simplicial group. This category encodes an involution compatible with an order-preserving multiplication. We introduce the \emph{category of involutive non-commutative sets}, which encodes the same information by adding data to the preimages of maps of finite sets. We demonstrate this category is isomorphic to $\HC$.

Hyperoctahedral homology is defined in terms of \emph{functor homology}, sometimes known as the \emph{homology of small categories}. In Section \ref{fun-hom-sec} we recall the necessary background on functor homology, including the tensor product of modules over a small category, its left derived functors and chain complexes that can be used to compute them.

In Section \ref{hyp-hom-sec} we define hyperoctahedral homology for involutive algebras in terms of the \emph{hyperoctahedral bar construction} following Fiedorowicz. We discuss the connection of hyperoctahedral homology to the other homology theories arising from crossed simplicial groups. We define \emph{hyperoctahedral homology with coefficients in a module} and prove a \emph{universal coefficient theorem}.

In Section \ref{aug-alg-sec} we prove that for an augmented, involutive algebra there is a smaller chain complex that computes hyperoctahedral homology. This smaller complex calculates hyperoctahedral homology using only the epimorphisms in the category $\HC$ and the elements of the augmentation ideal.

In Section \ref{group-hom-sec} we use the results of Section \ref{aug-alg-sec} to prove a connection between hyperoctahedral homology and the \emph{group homology of a product of hyperoctahedral groups} with coefficients in certain modules. We show that when the ground ring is a field of characteristic zero, hyperoctahedral homology of an augmented, involutive algebra can be calculated as the coinvariants of a group action.

In Section \ref{exten-sec} we show that by appending an initial object to the category $\HC$ we obtain a symmetric strict monoidal category. We extend the hyperoctahedral bar construction to be defined on this category and demonstrate that we can calculate hyperoctahedral homology using this more structured category.

In Section \ref{Fie-sec} we provide a proof of \emph{Fiedorowicz's theorem for the hyperoctahedral homology of a monoid algebra}, for a monoid with involution. This result, Theorem \ref{monoid-thm}, relates the hyperoctahedral homology of a monoid algebra to May's two-sided bar construction and the monads associated to the little intervals and little $\infty$-cubes operads.

In Section \ref{grp-alg-sec} we prove the main result of the paper, Theorem \ref{main-theorem}. Consider a discrete group of odd order. Consider its group algebra with involution induced by sending a group element to its inverse. We prove that the hyperoctahedral homology of such a group algebra is isomorphic to the homology of the fixed points under the involution of a $C_2$-equivariant infinite loop space built from the classifying space of the group.

\subsection*{Acknowledgements}
I would like to thank Sarah Whitehouse for all her help, support and guidance. I am very grateful to James Brotherston and James Cranch for their helpful suggestions.

\subsection*{Notational conventions}
Throughout the paper $k$ will denote a unital commutative ring. We let $\cyc=\llangle t\mid t^2=1\rrangle$ denote the cyclic group of order two. For $n\geqslant 0$ we let $\Sigma_{n+1}$ denote the symmetric group on the set $[n]=\llb 0,\dotsc , n\rrb$. By abuse of notation we will let $id_{n}$ denote both the identity element of $\Sigma_{n+1}$ and the unique order-preserving bijection $[n]\rightarrow [n]$. Unless otherwise stated, homology is taken with coefficients in the ground ring.

\section{The Hyperoctahedral Category and Involutive Non-commutative Sets}
\label{hyp-cat-sec}
We recall the definition of the hyperoctahedral groups $H_{n+1}$ for $n\geqslant 0$ and describe the associated category $\HC$, following \cite[Section 3]{FL}. We describe the category of involutive non-commutative sets, $\ifas$ and show that it is isomorphic to $\HC$.

\begin{defn}
For $n\geqslant 0$, the \emph{hyperoctahedral group} $H_{n+1}$ is defined to be the semi-direct product $ H_{n+1} \coloneqq  \cyc^{n+1} \rtimes \Sigma_{n+1}$ where $\Sigma_{n+1}$ acts on $\cyc^{n+1}$ by permuting the factors.
\end{defn}

The group $H_{n+1}$ is sometimes denoted as a wreath product $\Sigma_{n+1}\int \cyc$.
An element of the hyperoctahedral group $H_{n+1}$ is written as a tuple $\left(z_0,\dotsc ,z_n;\sigma\right)$ where each $z_i\in \cyc$ and $\sigma \in \Sigma_{n+1}$. Let $t_i$ denote the element $\left(1,\dotsc ,1,t,1,\dotsc ,1;id_{n}\right)$ in $H_{n+1}$, where $t$ occurs in the $i^{th}$ position for $0\leqslant i\leqslant n$. Let $\theta_j$ denote the element $\left(1,\dotsc ,1;\left( j\; j+1\right)\right)$ in $H_{n+1}$ for $0\leqslant j\leqslant n-1$. Methods for deriving a presentation for a semi-direct product of groups are well-known, see \cite[Chapter 10.2]{johnson-pres} for example. Following these methods one sees that the $H_{n+1}$ is generated by the elements of the form $t_i$ and $\theta_j$ where the $\theta_j$ satisfy the relations of the symmetric group $\Sigma_{n+1}$; the $t_i$ satisfy the relations of the group $C_2^{n+1}$; the $\theta_j$ and $t_i$ commute for $i<j$ and $i>j+1$; and $\theta_i \circ t_{i+1}=t_i\circ \theta_i$ and $\theta_i\circ t_i =t_{i+1}\circ \theta_i$.

The hyperoctahedral groups form a crossed simplicial group \cite[Section 3]{FL}. We recall the associated category $\HC$. Recall that the category $\Delta$ has as objects the sets $[n]=\llb 0,\dotsc , n\rrb$ for $n\geqslant 0$ and order-preserving maps as morphisms. For $n\geqslant 0$, $0\leqslant i\leqslant n+1$ and $0\leqslant j\leqslant n$, let $\delta_i \in \mathrm{Hom}_{\Delta}\left([n],[n+1]\right)$ be the unique order-preserving injection that omits $i\in [n+1]$ and let $\sigma_j\in \mathrm{Hom}_{\Delta}\left([n+1], [n]\right)$ be the unique order-preserving surjection satisfying $\sigma_j^{-1}(j)=\llb j ,j+1\rrb$. Recall that these \emph{face and degeneracy morphisms} generate the category $\Delta$ and are subject to the simplicial relations found in \cite[Appendix B]{Lod}, for example.

\begin{defn}
\label{delta-h-defn}
The category $\HC$ has as objects the sets $[n]=\llb 0,\dotsc , n\rrb$ for $n\geqslant 0$. An element of $\mathrm{Hom}_{\HC}\left([n],[m]\right)$ is a pair $\left(\phi , g\right)$ where $g \in H_{n+1}$ and $\phi\in \mathrm{Hom}_{\Delta}\left([n],[m]\right)$.

For $\left(\phi,g\right)\in \mathrm{Hom}_{\HC}\left([n],[m]\right)$ and $\left(\psi , h\right) \in \mathrm{Hom}_{\HC}\left([m],[l]\right)$ the composite is the pair $\left(\psi\circ h_{\star}(\phi), \phi^{\star}(h)\circ g\right)\in \mathrm{Hom}_{\HC}\left([n],[l]\right)$ where $h_{\star}(\phi)$ and $\phi^{\star}(h)$ are determined by the relations:
\[\left(\theta_k\right)_{\star}(\delta_i) =\delta_{\theta_k(i)}, \quad \left(\theta_k\right)_{\star}(\sigma_j) =\sigma_{\theta_k(j)}, \quad \left(t_k\right)_{\star}(\delta_i)=\delta_i, \quad \left(t_k\right)_{\star}(\sigma_j) =\sigma_{j} \] 
and
\begin{alignat*}{2}
     \begin{aligned}   \delta_i^{\star}\left(\theta_k\right)&=\begin{cases}
  \theta_k & k< i-1,\\
  id_{n} & k=i-1,\\
  id_{n} & k=i,\\
  \theta_{k-1} & k>i,\\
  \end{cases}\\
  \end{aligned}
   & \hskip 6em  &
  \begin{aligned}
   \sigma_j^{\star}(\theta_k)&= \begin{cases}
  \theta_k & k<j-1,\\
  \theta_j\theta_{j-1} & k=j-1,\\
  \theta_j\theta_{j+1} & k=j,\\
  \theta_{k+1} & k>j,\\
  \end{cases}\\ 
  \end{aligned}
  \end{alignat*}
\begin{alignat*}{2}
\begin{aligned}
   \delta_i^{\star}\left(t_k\right)&=\begin{cases}
  t_{k} & k< i,\\
  id_{n} & k=i,\\
  t_{k-1} & k>i,\\
  \end{cases} 
  \end{aligned}
    & \hskip 6em &
  \begin{aligned}
   \sigma_j^{\star}(t_k) &=\begin{cases}
  t_k & k<j,\\
  \theta_k t_{k+1}t_k & k=j,\\
  t_{k+1} & k>j,\\
  \end{cases} 
  \end{aligned}
\end{alignat*}
where the $\delta_i$ and $\sigma_j$ are the face and degeneracy maps of the category $\Delta$ and the $\theta_k$ and $t_k$ are the generators of the hyperoctahedral group.
\end{defn}

The hyperoctahedral category encodes the structure of an associative multiplication with a compatible involution. Intuitively speaking, the hyperoctahedral groups encode a total ordering via the symmetric group component and an involution via the cyclic group component whilst the order-preserving maps encode the multiplication.

The category of \emph{involutive non-commutative sets} encodes the same data by adding structure to the preimages of set maps analogously to the category of non-commutative sets of \cite[A10]{FT} and \cite[1.2]{Pir02}. Recent work of the author \cite{ifas} shows that the categories of involutive monoids and involutive bimonoids in a symmetric monoidal category are equivalent to categories of algebras over PROPs constructed from the category of involutive non-commutative sets. 

We provide an isomorphism between the hyperoctahedral category and the category of involutive non-commutative sets. The category of involutive non-commutative sets first appeared in the author's thesis \cite[Part \RNum{5}]{DMG}, where detailed examples and technical checks can be found.

We will denote the category of involutive, non-commutative sets by $\ifas$. It will have as objects the sets $[n]=\lbrace 0,\dotsc ,n\rbrace$ for $n\geqslant 0$. An element $f\in \mathrm{Hom}_{\ifas}\left([n],[m]\right)$ will be a map of sets such that the preimage of each singleton $i\in [m]$ is a totally ordered set such that each element comes adorned with a superscript label from the group $\cyc$.

\begin{rem}
Henceforth we will say that a morphism in $\ifas$ is a map of sets together with a \emph{labelled, ordered set} for each preimage. In particular, we will use \emph{preimage} to mean preimage of a singleton. We will denote composition in $\ifas$ by $\bullet$ in order to distinguish from the composition of maps of sets. In particular, we use $\circ$ for two morphisms in $\ifas$ if we are referring to the composite of the underlying maps of sets.
\end{rem}

\begin{defn}
We define an action of $\cyc$, which will be denoted by a superscript, on finite, ordered sets with $\cyc$-labels by
\[ \llb j_1^{\alpha_{j_1}}<\cdots <j_r^{\alpha_{j_r}}\rrb^t=\llb j_r^{t\alpha_{j_r}}<\cdots <j_1^{t\alpha_{j_1}}\rrb.\]
That is, we invert the ordering and multiply each label by $t\in \cyc$.
\end{defn}

\begin{defn}
\label{ifas-comp-defn}
Let $f_1\in \mathrm{Hom}_{\ifas}\left([n],[m]\right)$ and $f_2\in \mathrm{Hom}_{\ifas}\left([m],[l]\right)$. We define $f_2\bullet f_1\in \mathrm{Hom}_{\ifas}\left([n],[l]\right)$ to have underlying map of sets $f_2\circ f_1$.
We define the labelled totally ordered set $(f_2\bullet f_1)^{-1}(i)$ to be the ordered disjoint union of labelled, ordered sets
\[ \coprod_{j^{\alpha_j}\in f_2^{-1}(i)} f_1^{-1}\left(j\right)^{\alpha_j}.\]
\end{defn}

\begin{defn}
The \emph{category of involutive, non-commutative sets}, $\ifas$, has as objects the sets $[n]=\lbrace 0,\dotsc ,n\rbrace$ for $n\geqslant 0$. An element of $\mathrm{Hom}_{\ifas}([n],[m])$ is a map of sets with a total ordering on each preimage such that each element of the domain comes adorned with a superscript label from the group $\cyc$. Composition of morphisms is as defined in Definition \ref{ifas-comp-defn}.
\end{defn}

\begin{thm}
There is an isomorphism of categories $\HC \cong \ifas$.
\end{thm}
\begin{proof}
We note that $\ifas$ contains the morphisms of $\Delta$. These are the order-preserving maps of sets with the canonical total ordering on each preimage with each label being $1\in C_2$. Furthermore, $\ifas$ contains the elements of the hyperoctahedral groups. These are the bijections in $\ifas$. The isomorphism of categories is similar to the proof of \cite[3.11]{ifas}.
\end{proof}

\section{Functor homology}
\label{fun-hom-sec}
We recall some constructions from \emph{functor homology}, or the \emph{homology of small categories}, from \cite{Pir02} and \cite{GZ}.

Let $\C$ be a small category. We define the \emph{category of left $\mathbf{C}$-modules}, denoted $\CM$, to be the functor category $\mathrm{Fun}\left(\mathbf{C},\kmod\right)$. We define the \emph{category of right $\mathbf{C}$-modules}, denoted $\MC$, to be the functor category $\mathrm{Fun}\left(\mathbf{C}^{op},\kmod\right)$. It is well-known that the categories $\CM$ and $\MC$ are abelian with enough projectives and injectives, see for example \cite[Section 1.6]{Pir02}.

\begin{defn}
\label{triv-c-mod}
Let $k^{\star}$ denote the right $\C$-module that is constant at the trivial $k$-module. We will refer to this functor as the \emph{$k$-constant right $\C$-module}.
\end{defn}

\begin{defn}
\label{tensor-obj}
Let $G$ be an object of $\MC$ and $F$ be an object of $\CM$. We define the tensor product $G\otimes_{\C} F$ to be the $k$-module
\[\frac{\bigoplus_{C\in \mathrm{Ob}(\C)} G(C)\otimes_k F(C)}{\llangle G(\alpha)(x)\otimes y - x\otimes F(\alpha)(y)\rrangle}\]
where $\llangle G(\alpha)(x)\otimes y - x\otimes F(\alpha)(y)\rrangle$ is the $k$-submodule generated by the set
\[\llb G(\alpha)(x)\otimes y - x\otimes F(\alpha)(y): \alpha \in \mathrm{Hom}(\C),\,\, x\in src(G(\alpha)),\,\, y\in src(F(\alpha))\rrb.\]
\end{defn}

This quotient module is spanned $k$-linearly by equivalence classes of elementary tensors in $\bigoplus_{C\in \mathrm{Ob}(\C)} G(C)\otimes_k F(C)$ which we will denote by $\left[x\otimes y\right]$. One constructs the tensor product of $\C$-modules as a bifunctor $-\otimes_{\mathbf{C}}-\colon \MC \times \CM\rightarrow \kmod$ on objects by $(G,F)\mapsto G\otimes_{\C} F$. Given two natural transformations $\Theta \in \mathrm{Hom}_{\MC}\left(G, G_1\right)$ and $\Psi\in \mathrm{Hom}_{\CM}\left(F,F_1\right)$, the morphism $\Theta \otimes_{\C} \Psi$ is determined by $\left[x\otimes y\right] \mapsto \left[ \Theta_C(x)\otimes \Psi_C(y)\right]$. It is well-known that the bifunctor $-\otimes_{\C} -$ is right exact with respect to both variables and preserves direct sums, see for example \cite[Section 1.6]{Pir02}. We denote the left derived functors of $-\otimes_{\mathbf{C}}-$ by $\mathrm{Tor}_{\star}^{\mathbf{C}}(-,-)$.

Recall the nerve of $\C$ \cite[B.12]{Lod}. $N_{\star}\C$ is the simplicial set such that $N_n\C$ for $n\geqslant 1$ consists of all strings of composable morphisms of length $n$ in $\C$ and $N_0\C$ is the set of objects in $\C$. The face maps are defined to either compose adjacent morphisms in the string or truncate the string and the degeneracy maps insert identity morphisms into the string. We will denote an element of $N_n\C$ by $\left(f_n , \dotsc , f_1\right)$ where $f_i\in \mathrm{Hom}_{\C}\left(C_{i-1}, C_i\right)$.

For a small category $\C$ and a functor $F\in \CM$ there is a simplicial $k$-module, denoted $C_{\star}\left(\C , F\right)$ due to Gabriel and Zisman \cite[Appendix 2]{GZ} whose $n^{th}$ homology is canonically isomorphic to $\mr{Tor}_{n}^{\C}\left(k^{\star} , F\right)$.

\begin{defn}
\label{GabZisCpx}
Let $F\in \CM$. We define
\[C_n(\C,F)=\bigoplus_{(f_n,\dotsc ,f_1)} F(C_0)\]
where the sum runs through all elements $(f_n,\dotsc ,f_1)$ of $N_n\C$ and $f_i\in \mr{Hom}_{\C}\left(C_{i-1}, C_i\right)$. We write a generator of $C_n(\C,F)$ in the form $(f_n,\dotsc, f_1,x)$ where $(f_n,\dotsc ,f_1)\in N_n\C$ indexes the summand and $x\in F(C_0)$. 
The face maps $\p_i\colon C_n(\C,F)\rightarrow C_{n-1}(\C,F)$ are determined by
\[
\p_i(f_n,\dotsc, f_1,x)=\begin{cases}
(f_n,\dotsc ,f_2, F(f_1)(x)) & i=0,\\
(f_n, \dotsc , f_{i+1}\circ f_i,\dotsc ,f_1,x) & 1\leqslant i \leqslant n-1,\\
(f_{n-1},\dotsc ,f_1, x) & i=n.
\end{cases}\]
The degeneracy maps insert identity maps into the string.
By abuse of notation we will also denote the associated chain complex by $C_{\star}\left(\C , F\right)$. The homology of the associated chain complex will be denoted $H_{\star}\left(\C , F\right)$.
\end{defn}

\begin{rem}
\label{HSC-hocolim-rem}
This construction is analogous to the Bousfield-Kan construction \cite[\RNum{12}.2.1]{BK} for the homotopy colimit. This leads to an alternative choice of notation. We can denote the simplicial $k$-module $C_{\star}\left(\C , F\right)$ by
\[\underset{\C}{\mr{hocolim}}\, F.\]

This leads to an alternative description of the $k$-modules $\mathrm{Tor}_{\star}^{\C}\left(k^{\star} , F\right)$ in terms of \emph{derived colimits}. One can make the definition $\mathrm{colim}^{\C} (F)\coloneqq k\otimes_{\C} F$ following \cite[C.10]{Lod}. One then denotes the left derived functors by $\mathrm{colim}_i^{\C}(F)$.  We will make of use of this alternative notation in Section \ref{group-hom-sec}.
\end{rem}

There is an isomorphic variant of this chain complex constructed using the nerve of the under-category. For details of the under-category see \cite[\RNum{2}.6]{CWM} for instance. Let $\left(-\downarrow \C\right)\colon \C^{op}\rightarrow \mathbf{Cat}$ be the functor that sends an object $C\in \C$ to the under-category $\left(C\downarrow \C\right)$. For a morphism $f\in \mr{Hom}_{\C}\left(C , C^{\prime}\right)$, the functor $\left(f\downarrow \C\right) \colon \left(C^{\prime}\downarrow \C\right) \rightarrow \left(C\downarrow \C\right)$ is determined by precomposition with $f$. Let $k[-]\colon \mathbf{Set}\rightarrow \kmod$ denote the free $k$-module functor and, by abuse of notation, its extension to simplicial sets.

\begin{defn}
\label{nerve-complex}
Let $F \in \CM$. The chain complex $k\left[N_{\star}\left(-\downarrow \C\right)\right]\otimes_{\C} F$ has the $k$-module $k\left[N_{n}\left(-\downarrow \C\right)\right]\otimes_{\C} F$ in degree $n$ with the boundary map induced from the alternating sum of the face maps in the nerve. A generator in degree $n$ is an equivalence class $\left[\left( f_n,\dotsc , f_0\right) \otimes x\right]$ where $f_0\in \mathrm{Hom}_{\C}\left(C , C_0\right)$, $f_i\in \mathrm{Hom}_{\C}\left(C_{i-1}, C_i\right)$ for $i\geqslant 1$ and $x\in F(C)$.
\end{defn}

The isomorphism of chain complexes $k\left[N_{\star}\left(-\downarrow \C\right)\right]\otimes_{\C} F \cong C_{\star}\left( \C , F\right)$ for each $F\in \CM$ is determined by $\left[\left(f_n,\dotsc , f_0\right) \otimes x\right]\mapsto (f_n,\dotsc , f_1, F(f_0)(x))$ on generators in degree $n$. We will make use of both variants of this chain complex.

\section{Hyperoctahedral Homology} 
\label{hyp-hom-sec}
We define hyperoctahedral homology for involutive $k$-algebras in terms of functor homology. We describe the hyperoctahedral homology of the commutative ground ring, which is understood to have the trivial involution. We describe the relationship between hyperoctahedral homology and other homology theories arising from crossed simplicial groups. We prove a universal coefficient theorem for hyperoctahedral homology.

\begin{defn}
An \emph{involution} on an associative $k$-algebra $A$ is an anti-homomorphism of algebras $A\rightarrow A$, which we will denote by $a \mapsto \ol{a}$, which squares to the identity. A $k$-algebra equipped with an involution is called \emph{involutive}.
\end{defn}

\subsection{Hyperoctahedral bar construction}
We define the \emph{hyperoctahedral bar construction} following Fiedorowicz \cite[Section 2]{Fie}. Intuitively, the hyperoctahedral bar construction for an involutive $k$-algebra $A$ sends the object $[n]$ of $\HC$ to the tensor power $A^{\otimes n+1}$ and acts on a morphism $(\phi,g)\in\mr{Hom}_{\HC}\left([n],[m]\right)$ by sending it to the morphism of $k$-modules that permutes the factors of the tensor according to the underlying permutation of $g$, applies the involution according the labels of $g$ and multiplies tensor factors, or inserts identities, according to the order-preserving map $\phi$.

\begin{defn}
Let $A$ be an involutive, associative $k$-algebra with the involution denoted by $a \mapsto \ol{a}$. The \emph{hyperoctahedral bar construction} is the functor $\HB\colon \HC\rightarrow \kmod$
given on objects by $[n]\mapsto A^{\otimes n+1}$. Let $(\phi ,g )\in \mr{Hom}_{\HC}\left([n],[m]\right)$. Then $\phi \in \mr{Hom}_{\Delta}\left([n] , [m]\right)$ and $g=\left( z_0,\dotsc , z_n; \sigma\right) \in H_{n+1}$. We define
\[\HB\left(\phi ,g\right)(a_0\otimes \cdots \otimes a_n)=\left( \underset{_{i\in (\phi\circ \sigma)^{-1}(0)}}{{\prod}^{<}} a_i^{z_i}\right)\otimes \cdots \otimes \left(\underset{_{i\in (\phi\circ \sigma)^{-1}(m)}}{{\prod}^{<}} a_i^{z_i} \right)\]
on elementary tensors and extend $k$-linearly, where the product is ordered according to the map $\phi$ and
\[a^{z_i}=
\begin{cases}
a & z_i=1\\
\ol{a} & z_i=t.
\end{cases}\]
Note that an empty product is defined to be the multiplicative unit $1_A$.
\end{defn}

\begin{rem}
We note that the notation for the hyperoctahedral bar construction presented here differs from the notation $B_A^{oct}$ used in Fiedorowicz's preprint \cite{Fie} and the author's thesis \cite{DMG}. We hope the reader will agree that our choice of notation offers equal clarity ($\mathsf{H}$ for ``hyperoctahedral bar construction" with a subscript $A$ to indicate the algebra under consideration) whilst making the typesetting for the remainder of the paper considerably more pleasant. Proofs that the hyperoctahedral bar construction is well-defined can be found in \cite[Lemma 2.1]{Fie} or \cite[Appendix B]{DMG}.
\end{rem}

\subsection{Hyperoctahedral homology}
Recall the $k$-constant right $\HC$-module $k^{\star}$ from Definition \ref{triv-c-mod}.

\begin{defn}
Let $A$ be an involutive, associative algebra. For $n\geqslant 0$, we define the \emph{$n^{th}$ hyperoctahedral homology of $A$} to be
\[HO_{n}(A)\coloneqq\mr{Tor}_{n}^{\HC}\left(k^{\star},\mathsf{H}_A\right).\]
\end{defn}

\begin{prop}
The hyperoctahedral homology of the ground ring $k$, where $k$ has the trivial involution, is isomorphic to $k$ concentrated in degree zero.
\end{prop}
\begin{proof}
The key is to prove that the nerve of the category $\HC$ is contractible. By \cite[Example 6]{FL}, the hyperoctahedral crossed simplicial group $\llb H_{n+1}\rrb$ is contractible as a simplicial set with the inclusion maps $H_i\rightarrow H_{i+1}$ forming a contracting homotopy. By \cite[Proposition 5.8]{FL} it follows that $\Omega B\HC$, the based loop space on the classifying space of $\HC$, is a contractible space. By considering the long exact sequence in homotopy associated to the path space fibration we deduce that $B\HC$ is a contractible space. By definition the classifying space $B\HC$ is the geometric realization of the nerve $N_{\star}\HC$. We deduce that $N_{\star}\HC$ is contractible and, in particular, its homology is isomorphic to $k$ concentrated in degree zero.

The chain complex $C_{\star}\left(\HC, \mathsf{H}_k\right)$, defined following Definition \ref{GabZisCpx} is generated $k$-linearly in degree $p$ by elements of the form $\left(f_p,\dotsc , f_1, 1_k\otimes \cdots \otimes 1_k\right)$ where $f_i\in \mathrm{Hom}_{\HC}\left([x_{i-1}],[x_i]\right)$. 

We observe that there is an isomorphism of chain complexes $C_{\star}\left( \HC,\mathsf{H}_k\right)\rightarrow k\left[N_{\star}(\HC)\right]$ determined on generators in degree $p$ by
$\left(f_p,\dotsc , f_1, 1_k\otimes \cdots \otimes 1_k\right)\mapsto \left(f_p,\dotsc , f_1\right)$, from which the result follows.
\end{proof}

\subsection{Relationship to dihedral homology}
Dihedral homology for algebras with involution was first studied independently by Loday \cite{Lod-dihed} and Krasauskas, Lapin and Solov'ev \cite{KLS-dihed}. It is the homology theory associated to the dihedral crossed simplicial group \cite[Example 5]{FL} and is related to $O(2)$-equivariant homology \cite[Theorem 3.3.3]{Lodder1} and Hermitian algebraic $K$-theory via a dihedral Chern character \cite[Section 4]{KLS-dihed}. 

The category $\Delta D$ associated to the dihedral crossed simplicial group is self-dual, that is $\Delta D\cong \Delta D^{op}$, see for example \cite[Proposition 1.4]{Dunn}. Let $A$ be an involutive $k$-algebra. One observes, similarly to \cite[Lemma 2.2]{Fie}, that the composite
\[ \Delta D^{op}\xra{\cong} \Delta D\hookrightarrow \HC \xra{\HB} \kmod\]
is the dihedral bar construction $A^{\dagger}$ of \cite[Definition 1.4.1]{Lodder1}.

Let $\mathsf{D}_A$ denote the restriction of $\HB$ to the subcategory $\Delta D$. Recall the $k$-constant right $\C$-module from Definition \ref{triv-c-mod}. By abuse of notation we will use $k^{\star}$ to denote both the $k$-constant right $\Delta D^{op}$-module and the $k$-constant right $\Delta D$-module. We observe that
\[HD_{\star}(A)=\mr{Tor}_{\star}^{\Delta D^{op}}\left(k^{\star}, A^{\dagger}\right)\cong \mr{Tor}_{\star}^{\Delta D}\left(k^{\star} , \mathsf{D}_A\right).\]

It follows that for an involutive $k$-algebra $A$, there is a natural map $HD_{\star}\left(A\right) \rightarrow HO_{\star}\left(A\right)$.

We can obtain similar maps for symmetric homology and cyclic homology. For an algebra with involution $A$, the restriction of the hyperoctahedral bar construction $\HB$ to the category $\DS$ associated to the symmetric crossed simplicial group is $\SB$, the symmetric bar construction of \cite[Definition 13]{Ault}. We therefore have a natural map $HS_{\star}\left(A\right) \rightarrow HO_{\star}\left(A\right)$, constructed similarly to the dihedral case. Pre-composing with the natural map $HC_{\star}\left(A\right)\rightarrow HS_{\star}\left(A\right)$ of \cite[Section 12]{Ault} yields, for an involutive $k$-algebra, a natural map $HC_{\star}\left(A\right)\rightarrow HO_{\star}\left(A\right)$.

\subsection{Universal coefficient theorem}
We show that hyperoctahedral homology satisfies a universal coefficient theorem analogous to that satisfied by symmetric homology \cite[Theorem 21]{Ault}. 

\begin{defn}
\label{coeffs-defn}
Let $A$ be an involutive $k$-algebra and let $M$ be a $k$-module. Let $d$ denote the boundary map of $C_{\star}\left(\HC , \HB\right)$. We define the \emph{hyperoctahedral homology of $A$ with coefficients in $M$}, denoted $HO_{\star}\left(A;M\right)$, to be the homology of the chain complex $C_{\star}\left(\HC , \HB\right) \otimes M$ with boundary map $d\otimes id_M$. 
\end{defn}

Let $A$ be an involutive $k$-algebra. One observes that for each $n\geqslant 0$, $HO_{n}\left(A;-\right)$ is a covariant endofunctor on the category of $k$-modules. Explicitly, let $M$ be a $k$-module thought of as a chain complex concentrated in degree zero. The functor $HO_{n}\left(A;-\right)$ is defined by first taking the tensor product of chain complexes $C_{\star}\left(\HC , \HB\right)\otimes -$ and then taking the $n^{th}$ homology. Since the tensor product of $k$-modules and taking homology both commute with direct sums we see that the functors $HO_{n}\left(A;-\right)$ are additive. A standard argument shows that if $A$ is flat over $k$, given a short exact sequence of $k$-modules, the functors $HO_{n}\left(A;-\right)$ give rise to a long exact sequence in homology.

\begin{thm}
\label{UCT}
Let $A$ be a flat, involutive $k$-algebra and let $M$ be a $k$-module. There is a spectral sequence with
\[E_2^{p,q}\cong \mr{Tor}_p^k\left(HO_q\left(A\right), M\right)\]
converging to $HO_{\star}\left(A;M\right)$.
\end{thm}
\begin{proof}
This is the ``universal coefficient spectral sequence" of \cite[Theorem 12.11]{SS}. We note that the boundedness condition is satisfied since $HO_q\left(A;-\right)=0$ for $q< 0$. The tensor product of $k$-modules commutes with arbitrary direct sums as does taking homology and so the functors $HO_q\left(A;-\right)$ commute with arbitrary direct sums.
\end{proof}

We recall some definitions from homological algebra \cite[4.1, 4.2]{weib}. Recall that the \emph{weak global dimension} of a commutative ring $k$ is the largest integer $n$ such that $\mathrm{Tor}_n^k\left(M,N\right)$ is non-zero for some $k$-modules $M$ and $N$. We will denote it $wgd(k)$. Furthermore, recall that a commutative ring $k$ is said to be \emph{hereditary} if every ideal is projective.

\begin{cor}
If the ground ring $k$ satisfies $wgd(k)\leqslant 1$ then the spectral sequence of Theorem \ref{UCT} reduces to short exact sequences
\[0\rightarrow HO_n(A)\otimes M \rightarrow HO_n\left(A;M\right) \rightarrow \mathrm{Tor}_1^k\left(HO_{n-1}\left(A\right), M\right)\rightarrow 0.\]
Furthermore, if $k$ is hereditary and $A$ is projective these short exact sequences split, although not naturally.
\end{cor}
\begin{proof}
By definition, if $wgd(k)\leqslant 1$ then $\mr{Tor}_p^k\left(HO_q\left(A\right), M\right)=0$ for all $p>1$. The corollary now follows from \cite[Korollar 2.13]{Dold}. If $k$ is hereditary and $A$ is projective the splitting follows from \cite[Corollary 7.56]{Rotman}.
\end{proof}

\section{Hyperoctahedral Homology of an Augmented Algebra}
\label{aug-alg-sec}
We prove that in the case of an augmented involutive $k$-algebra there exists a smaller chain complex that computes hyperoctahedral homology. We will do this in two stages. We will identify a subcomplex of the standard complex which can be used to calculate the hyperoctahedral homology of an augmented, involutive $k$-algebra up to a copy of $k$ in degree zero. This subcomplex contains only those elements of $A$ that lie in the augmentation ideal. We will then introduce the epimorphism construction for hyperoctahedral homology. Using this construction we reduce the complex further, showing that we need only use the epimorphisms in $\HC$ to calculate the hyperoctahedral homology. The material in this section first appeared in the author's thesis \cite[Part \RNum{6}]{DMG} and was inspired by the analogous construction for symmetric homology \cite[Section 6]{Ault}. Throughout this section we will denote an element of $N_p\left([x]\downarrow \HC\right)$ by $\left(f_p,\dotsc , f_1, f_0\right)$. In particular, $f_0=\left(\phi , g\right)\in \mathrm{Hom}_{\HC}\left([x],[x_0]\right)$ and $f_i=\left(\phi_i,g_i\right)\in \mathrm{Hom}_{\HC}\left([x_{i-1}],[x_i]\right)$ for $i\geqslant 1$.

\subsection{Reduced hyperoctahedral homology}
\label{red-HO}

An associative $k$-algebra $A$ is said to be \emph{augmented} if it is equipped with a $k$-algebra homomorphism $\varepsilon\colon A\rightarrow k$. The \emph{augmentation ideal} $I$ of $A$ is defined to be $\mathrm{Ker} (\ve)$. Note that the structure of an augmented algebra is compatible with an involution. In particular, the augmentation ideal is closed under involution. Recall from \cite[Section 1.1.1]{LodVal} that for an augmented $k$-algebra there is an isomorphism of $k$-modules $A\cong I\oplus k$. It follows that every element $a\in A$ can be written uniquely in the form $y+\lambda$ where $y\in I$ and $\lambda \in k$.

\begin{defn}
\label{basic-tens-defn}
Let $A$ be an augmented $k$-algebra with augmentation ideal $I$. A \emph{basic tensor} in $A^{\otimes n}$ is an elementary tensor $a_1\otimes \cdots \otimes a_n$ such that either $a_i\in I$ or $a_i=1_k$ for each $1\leqslant i\leqslant n$. A tensor factor $a_i$ is called \emph{trivial} if $a_i=1_k$ and is called \emph{non-trivial} if $a_i\in I$.
\end{defn}

Note that for an augmented $k$-algebra $A$, the $k$-module $A^{\otimes n}$ is generated by all finite $k$-linear combinations of its basic tensors.

Henceforth in this section $A$ will denote an augmented, involutive $k$-algebra with augmentation ideal $I$. Consider the chain complex $k\left[N_{\star}\left(-\downarrow \HC\right)\right]\otimes_{\HC} \HB$ defined analogously to Definition \ref{nerve-complex}, whose homology is the hyperoctahedral homology of $A$. By analysing the quotient module one can show that $k\left[N_{\star}\left(-\downarrow \HC\right)\right]\otimes_{\HC} \HB$ is generated $k$-linearly in degree $n$ by those equivalence classes of the form $\left[\left(f_n,\dotsc ,f_1, f_0\right)\otimes Y_x\right]$ such that $Y_x$ is a basic tensor containing no trivial factors and those equivalence classes of the form $\left[\left(f_n,\dotsc ,f_1, f_0\right)\otimes 1_k\right]$. Note that for this last set of equivalence classes the domain of $f_0$ is the set $[0]$. We observe that the boundary map of $k\left[N_{\star}\left(-\downarrow \HC\right)\right]\otimes_{\HC} \HB$ preserves these two types of equivalence classes. We can therefore make the following definition.

\begin{defn}
\label{complex-defn}
Let $C_{\star}\left(\HC , I\right)$ denote the subcomplex of $k\left[N_{\star}\left(-\downarrow \HC\right)\right]\otimes_{\HC} \HB$ generated $k$-linearly in degree $n$ by all equivalence classes $\left[\left(f_n,\dotsc ,f_1, f_0\right)\otimes Y_x\right]$
such that $Y_x$ is a basic tensor containing no trivial factors. We denote the homology of this subcomplex by $H_{\star}\left(\HC ,I\right)$. Let $C_{\star}\left(\HC , k\right)$ denote the subcomplex of $k\left[N_{\star}\left(-\downarrow \HC\right)\right]\otimes_{\HC} \HB$ generated $k$-linearly in degree $n$ by all equivalence classes of the form $\left[\left(f_n,\dotsc ,f_1, f_0\right)\otimes 1_k\right]$.
\end{defn} 

\begin{prop}
Let $A$ be an involutive augmented $k$-algebra with augmentation ideal $I$. For each $n\geqslant 1$ there is an isomorphism of $k$-modules $HO_{n}\left(A\right)\cong H_n\left(\HC , I\right)$. Furthermore, there is an isomorphism of $k$-modules $HO_0\left(A\right)\cong H_0\left(\HC , I\right) \oplus k$.
\end{prop}
\begin{proof}
The inclusion of the subcomplex $C_{\star}\left(\HC , I\right)$ into $k\left[N_{\star}\left(-\downarrow \HC\right)\right]\otimes_{\HC} \HB$ has a left inverse given by the canonical projection map, the kernel of which is $C_{\star}\left(\HC , k\right)$. This gives a splitting of $k\left[N_{\star}\left(-\downarrow \HC\right)\right]\otimes_{\HC} \HB$ as a direct sum $C_{\star}\left(\HC , I\right) \oplus C_{\star}\left(\HC , k\right)$. In order to prove the result we will invoke \cite[8.4.6]{weib} to show that the homology of $C_{\star}\left(\HC , k\right)$ is isomorphic to $k$ concentrated in degree zero. One defines an augmentation $\varepsilon\colon C_0\left(\HC , k\right) \rightarrow k$ by $\left[f\otimes 1_k\right]\mapsto 1_k$. One then checks that the maps $h_n\colon C_n\left(\HC,k\right)\rightarrow C_{n+1}\left(\HC,k\right)$ determined by $\left[\left(f_n,\dotsc ,f_1,f\right)\otimes 1_k\right] \mapsto \left[\left(f_n,\dotsc ,f_1,f, id_{0}\right)\otimes 1_k\right]$ for $n\geqslant 0$ and $h_{-1}\colon k \rightarrow C_0\left(\HC , k\right)$ determined by $1_k \mapsto \left[id_{0}\otimes 1_k\right]$ satisfy the conditions described in \cite[8.4.6]{weib}.
\end{proof}

\begin{defn}
Let $A$ be an augmented, involutive $k$-algebra. For $n\geqslant 0$, we define the \emph{$n^{th}$ reduced hyperoctahedral homology of $A$} to be
\[\widetilde{HO}_{n}(A)\coloneqq H_{n}\left( \HC ,I\right).\]
\end{defn}

\subsection{Epimorphism construction}
We begin by noting that a morphism $\left(\phi , g\right)$ in the category $\HC$ is an epimorphism if and only if the underlying map of sets is surjective.

\begin{rem}
\label{decomp-rem}
Consider a morphism $(\varphi , g)\in \mr{Hom}_{\HC}\left([x],[z]\right)$ and let $r=\llv \mr{Im}\left(\varphi\right)\rrv$. By the unique decomposition of morphisms in $\Delta$ \cite[Theorem B.2]{Lod} the morphism $\left(\varphi , g\right)$ can be written uniquely in the form $\left(i_{\phi}, id_{r-1}\right)\circ \left(\pi_{\phi},g\right)$ where $\left(\pi_{\phi},g\right) \in \mathrm{Hom}_{\EDH}\left([x],[r-1]\right)$ and $i_{\phi}\colon [r-1]\rightarrow [z]$ is a monomorphism in the category $\Delta$. Whenever we use this decomposition we will abuse notation and write $i_{\phi}$ for $\left(i_{\varphi},id_{r-1}\right)\in \mr{Hom}_{\HC}\left([r-1], [z]\right)$.
\end{rem}

\begin{defn}
Let $\EDH$ denote the subcategory of $\HC$ whose morphisms are the epimorphisms in $\HC$.
\end{defn}

\begin{defn}
Let $A$ be an augmented, involutive $k$-algebra with augmentation ideal $I$. We define a functor $\HBI\colon\EDH\rightarrow \kmod$ on objects by $[n]\rightarrow I^{\otimes n+1}$. On morphisms we define $\HBI$ to be the restriction of $\HB$ to the subcategory $\EDH$.
\end{defn}

\begin{rem}
The functor $\HBI$ is well-defined since we are only taking the epimorphisms of $\HC$ and the augmentation ideal $I$ is closed under the involution and under multiplication. 
\end{rem}

We define the chain complex $k\left[N_{\star}\left(-\downarrow \EDH\right)\right] \otimes_{\EDH} \HBI$ following Definition \ref{nerve-complex}. We note that there is an inclusion of chain complexes
\[i~_{\star}\colon k\left[N_{\star}\left(-\downarrow \EDH\right)\right] \otimes_{\EDH} \HBI\rightarrow C_{\star}\left(\HC , I\right).\]
We claim that this inclusion is a chain homotopy equivalence. In order to prove this we introduce the \emph{epimorphism construction for hyperoctahedral homology}.

\begin{lem}
\label{epi-induce}
A morphism 
\[\left(\psi , h\right) \colon \left([x]\xra{(\phi_1 ,g_1)} [z_1]\right) \rightarrow \left( [x]\xra{(\phi_2,g_2)} [z_2]\right)\]
in $\left([x]\downarrow\HC\right)$ induces a unique morphism
\[\ol{\left(\psi ,h\right)} \colon \left([x]\xra{(\pi_{\phi_1} ,g_1)} [r_1-1]\right) \rightarrow \left( [x]\xra{(\pi_{\phi_2},g_2)} [r_2-1]\right)\]
in $\left([x]\downarrow \EDH\right)$.
\end{lem}
\begin{proof}
Using Remark \ref{decomp-rem}, we have a commutative diagram
\begin{small}
\begin{center}
\begin{tikzcd}
\lb z_1\rb\arrow[rrrrrrrr, "{(\psi ,h)}", bend left] &&&& \lb x\rb\arrow[llll, "{(\phi_1 ,g_1)}",swap]\arrow[rrrr, "{(\phi_2 ,g_2)}"]\arrow[dll, "{(\pi_{\phi_1} , g_1)}"]\arrow[drr, "{(\pi_{\phi_2}, g_2)}",swap] &&&& \lb z_2\rb\\
&& \lb r_1-1\rb\arrow[ull, "{i_{\phi_1}}"]&&&& \lb r_2-1\rb\arrow[urr, "{i_{\phi_2}}",swap]
\end{tikzcd}
\end{center}
\end{small}
in $\HC$.

By construction, $\mathrm{Im}\left(i_{\phi_1}\right)= \mathrm{Im}\left(\left(\phi_1 , g_1\right)\right)$. From the commutativity of the diagram one deduces that $\mathrm{Im}\left((\psi ,h)\circ i_{\phi_1}\right)= \mathrm{Im}\left(i_{\phi_2}\right)$.

By Remark \ref{decomp-rem} there exists a unique morphism $\ol{(\psi,h)}\colon [r_1-1]\rightarrow [r_2-1]$ in $\EDH$ such that $(\psi ,h)\circ i_{\phi_1}= i_{\phi_2}\circ \ol{(\psi,h)}$.

By commutativity of the diagram we see that $i_{\phi_2}\circ \ol{(\psi,h)}\circ(\pi_{\phi_1} , g_1)= i_{\phi_2}\circ(\pi_{\phi_2}, g_2)$ and since $i_{\phi_2}$ is a monomorphism in $\HC$ we can cancel on the left and deduce that 
$\ol{(\psi,h)}\circ(\pi_{\phi_1} , g_1)=(\pi_{\phi_2}, g_2)$ as required.
\end{proof}

A routine check, which can be found in \cite[Proposition 27.8.5]{DMG}, demonstrates that this assignment on morphisms is functorial. We can therefore make the following definition.

\begin{defn}
\label{epi-const-func}
For $x\geqslant 0$, we define the functor $E_x\colon \left([x]\downarrow\HC\right)\rightarrow \left([x]\downarrow \EDH\right)$ on objects by $\left(\phi , g\right)\mapsto \left(\pi_{\phi} ,g\right)$ and on morphisms by $\left(\psi ,h\right)\mapsto \ol{\left(\psi ,h\right)}$.
We call $E_x$ the \emph{epimorphism construction for hyperoctahedral homology}.
\end{defn}

We observe that by the functoriality of the nerve construction, the functor $E_x$ induces a map of simplicial sets $N_{\star}E_x\colon N_{\star}([x]\downarrow\HC)\rightarrow N_{\star}([x]\downarrow \EDH)$. This induced map takes a string of composable morphisms in $\HC$ and sends it to the string of composable morphisms in $\EDH$ obtained by applying the epimorphism construction.

\begin{defn}
We define a map of chain complexes 
\[\chi_{\star}\colon C_{\star}\left( \HC,I\right)\rightarrow k\left[N_{\star}(-\downarrow\EDH)\right]\otimes_{\EDH} \HBI\]
to be determined in degree $n$ by 
\[\left[\left( f_n,\dotsc , f_1, f_0\right)\otimes Y_x\right]\mapsto \left[k\left[N_nE_x\right]\left( f_n,\dotsc ,f_1,f_0\right)\otimes Y_x\right].\]
\end{defn}

\begin{rem}
One must check that the maps $\chi_n$ are well-defined. This consists of a routine but lengthy check that the epimorphism construction is compatible with equivalence classes under the tensor products $-\otimes_{\HC} -$ and $-\otimes_{\EDH}-$. The full details can be found in \cite[Proposition 27.9.1]{DMG}.
\end{rem} 

\begin{thm}
\label{epi-thm}
Let $A$ be an augmented, involutive $k$-algebra with augmentation ideal $I$. For $n\geqslant 0$ there exist isomorphisms of $k$-modules
\[\widetilde{HO}_{n}(A) \cong H_{n}\left(k\left[N_{\star}(-\downarrow\EDH)\right]\otimes_{\EDH} \HBI\right)\cong H_{n}\left(C_{\star}\left(\EDH ,\HBI\right)\right).\]
\end{thm}
\begin{proof}
We show that the inclusion map of chain complexes
\[i_{\star}\colon k\left[N_{\star}\left(-\downarrow \EDH\right)\right] \otimes_{\EDH} \HBI\rightarrow C_{\star}\left(\HC , I\right)\]
is a chain homotopy equivalence with inverse given by $\chi_{\star}$.
By construction, the composite $\chi_{\star} \circ i_{\star}$ is equal to the identity map on the chain complex $k\left[N_{\star}\left(-\downarrow \EDH\right)\right] \otimes_{\EDH} \HBI$. One shows that the composite $i_{\star}\circ \chi_{\star}$ is homotopic to the identity map on the chain complex $C_{\star}\left(\HC , I\right)$ by constructing a presimplicial homotopy between the two maps. We will make use of Remark \ref{decomp-rem}.

Let $h_j^n\colon C_n\left(\HC ,I\right) \rightarrow C_{n+1}\left(\HC ,I\right)$ be determined by
\[ \left[\left(f_n,\dotsc , f_1,f_0\right)\otimes Y_x\right] \mapsto \left[\left(f_n,\dotsc , f_{j+1},i_{\phi_j}, \ol{f_j},\dotsc , \ol{f_1}, E_{x}(f_0)\right)\otimes Y_x\right].\]
That is, $h_j^n$ is the map determined by mapping an equivalence class in $C_n\left(\HC ,I\right)$ to the equivalence class in $C_{n+1}\left(\HC ,I\right)$ obtained by applying the epimorphism construction to the first $j$ morphisms in the representative, inserting $i_{\phi_j}$ from the unique decomposition of Remark \ref{decomp-rem} and leaving the remaining morphisms unchanged. A routine check, which can be found in \cite[Proposition 27.10.2]{DMG}, demonstrates that these maps do indeed form a presimplicial homotopy between $i_{\star}\circ \chi_{\star}$ and the identity map.
\end{proof}

\section{Hyperoctahedral Homology and Group Homology}
\label{group-hom-sec}
We provide an application of our smaller complex for computing the hyperoctahedral homology of an augmented, involutive algebra $A$. We will show that in this case we can express the reduced hyperoctahedral homology of $A$ in terms of the group homology of a product of hyperoctahedral groups with coefficients in modules constructed from the epimorphisms in $\HC$. When the ground ring is a field of characteristic zero one can say more: the reduced hyperoctahedral homology of $A$ can be computed as the coinvariants of a group action. The key to this construction is S\l{}omi\'{n}ska's work on E-I-A-categories \cite{slom}, that is categories for which all endomorphisms are isomorphisms and all isomorphisms are automorphisms. We observe that $\EDH$ is just such a category. The results in this section are analogous to the symmetric homology case \cite[Section 7]{Ault}. The notation throughout this section is chosen to closely resemble \cite{slom} and \cite[Section 7]{Ault}.

\begin{defn}
Let $\mathbf{S}_0$ denote the opposite category of the poset of non-empty finite subsets of $\mathbb{N}$. The set of objects are non-empty ordered tuples of non-negative integers $\left(x_r<\cdots <x_0\right)$. There is a unique morphism $\left(x_r<\cdots <x_0\right)\rightarrow \left(x_{r^{\prime}}^{\prime}<\cdots <x_0^{\prime}\right)$ if $\llb x_0^{\prime},\dotsc , x_{r^{\prime}}^{\prime}\rrb \subseteq \llb x_0,\dotsc , x_r\rrb$.
\end{defn}

\begin{defn}
We define a functor $\mathcal{A}\colon \mathbf{S}_0\rightarrow \mathbf{Grp}$, where $\mathbf{Grp}$ is the category of groups, on objects by
\[\mathcal{A}\left(x_r<\cdots <x_0\right)= \prod_{i=0}^r \mathrm{Aut}_{\EDH} \left([x_i]\right) = \prod_{i=0}^r H_{x_i+1}.\]
The functor $\mathcal{A}$ acts on a morphism of $\mathbf{S}_0$ by projecting the factors of the product corresponding to the elements of the codomain.
\end{defn}

\begin{defn}
We define a functor $\mathcal{E}\colon \mathbf{S}_0\rightarrow \mathbf{Set}$ on objects by
\[\mathcal{E}\left(x_r<\cdots <x_0\right)=
\begin{cases}
\llb x_0\rrb & r=0\\
\prod_{i=1}^r \mathrm{Hom}_{\EDH}\left( [x_{i-1}] , [x_i]\right) & n\geqslant 1,
\end{cases}\]
where $\llb x_0\rrb$ is the one element set containing $x_0$. For a morphism $f\colon  \left(x_r<\cdots <x_0\right)\rightarrow \left(x_{r^{\prime}}^{\prime}<\cdots <x_0^{\prime}\right)$, $\mathcal{E}(f)$ is determined by composing adjacent morphisms in the product (or projecting factors in the case of $x_0$ and $x_r$) according to the elements present in the codomain. If the codomain is a tuple with one element, $x_0^{\prime}$, then $\mathcal{E}(f)$ is the unique set map into $\llb x_0^{\prime}\rrb$.
\end{defn}

\begin{defn}
\label{action-defn}
Let $\left(x_r<\cdots <x_0\right)$ be an object of $\mathbf{S}_0$. We define an action $\mu$ of the group $\mathcal{A}\left(x_r<\cdots <x_0\right)$ on the set $\mathcal{E}\left(x_r<\cdots <x_0\right)$ by
\[ \left(g_n, \dotsc ,g_0\right) \times \left(f_n,\dotsc , f_1\right)\mapsto \left(g_nf_ng_{n-1}^{-1}, \dotsc , g_1f_1g_0^{-1}\right).\]
\end{defn}

\begin{rem}
\label{action-rem}
We observe that this action is compatible with morphisms in $\mathbf{S}_0$. That is, for a morphism $f\colon  \left(x_r<\cdots <x_0\right)\rightarrow \left(x_{r^{\prime}}^{\prime}<\cdots <x_0^{\prime}\right)$, we have an equality $\mu \circ \left(\mathcal{A}(f)\times \mathcal{E}(f)\right) = \mathcal{E}(f) \circ \mu$.
\end{rem}

\begin{defn}
Let $A$ be an augmented, involutive $k$-algebra with augmentation ideal $I$. We define a functor $\mathcal{E}_I\colon \mathbf{S}_0\rightarrow \kmod$ on objects by
\[\mathcal{E}_I\left(x_r<\cdots <x_0\right) = k\left[\mathcal{E}\left(x_r<\cdots <x_0\right)\right] \otimes_k \HBI\left([x_0]\right).\]
Let $f\colon \left(x_r<\cdots <x_0\right)\rightarrow \left(x_{r^{\prime}}^{\prime}<\cdots <x_0^{\prime}\right)$ be a morphism in $\mathbf{S}_0$ and suppose $x_0^{\prime}= x_j$. The morphism $\mathcal{E}_I(f)$ is determined by
\[\left(f_n,\dotsc , f_1\right) \otimes x \mapsto \mathcal{E}(f)\left(f_n,\dotsc , f_1\right) \otimes \HBI\left(f_j\circ \cdots \circ f_1\right)(x).\]
\end{defn}

\begin{rem}
Observe that $k\left[\mathcal{E}\left(x_r<\cdots <x_0\right)\right] \otimes_k \HBI\left([x_0]\right)$ is an $\mathcal{A}\left(x_r<\cdots <x_0\right)$-module via the action of Definition \ref{action-defn} and that this construction is functorial by Remark \ref{action-rem}. We therefore have well-defined functors $H_q\left(\mathcal{A}(-), \mathcal{E}_I(-)\right) \colon \mathbf{S}_0\rightarrow \kmod$ for $q\geqslant 0$. 
\end{rem}

We can express these group homology functors in terms of the homology of small categories. Recall the notation of Remark \ref{HSC-hocolim-rem}. We fix $X=\left(x_r<\cdots <x_0\right)\in \mathbf{S}_0$. Let $GX$ be the category with one object and endomorphism set $\mathcal{A}(X)$. We observe that
\[H_q\left(\mathcal{A}(X), \mathcal{E}_I(X)\right) = H_q\left(GX , \mathcal{E}_I\right) = \mathrm{colim}_q^{GX} \mathcal{E}_I.\]

\begin{thm}
\label{ss-thm}
Let $A$ be an augmented, involutive $k$-algebra with augmentation ideal $I$. There exists a spectral sequence with
\[E_2^{p,q}= \mathrm{colim}_p^{\mathbf{S}_0} H_q\left(\mathcal{A}(-),\mathcal{E}_I(-)\right) \Rightarrow \widetilde{HO}_{p+q}(A).\]
\end{thm}
\begin{proof}
Recall the notation of Remark \ref{HSC-hocolim-rem} and consider
\[Y= \underset{\mathbf{S}_0}{\mr{hocolim}}\, \underset{GX}{\mathrm{hocolim}}\, \mathcal{E}_I.\]
This is a bisimplicial $k$-module with the horizontal simplicial structure induced from the nerve of $GX$ and the vertical simplicial structure induced from the nerve of $\mathbf{S}_0$. Taking the ``horizontal homology first" spectral sequence associated to this bisimplicial $k$-module we obtain a spectral sequence of the form
\[E_2^{p,q}= \mathrm{colim}_p^{\mathbf{S}_0} \mathrm{colim}_q^{GX} \mathcal{E}_I \Rightarrow H_{p+q}(Y).\]
By \cite[Proposition 1.6]{slom} and Theorem \ref{epi-thm} we have 
\[ H_{p+q}(Y)\cong \mathrm{colim}_{p+q}^{\EDH} \HBI = \widetilde{HO}_{p+q}(A)\]
as required.
\end{proof}

\begin{cor}
Suppose that $k$ is a field of characteristic zero. Then
\[\widetilde{HO}_p(A) = \mathrm{colim}_p^{\mathbf{S}_0} \, \mathcal{E}_I(X)_{\mathcal{A}(X)},\]
where $\mathcal{E}_I(X)_{\mathcal{A}(X)}$ denotes the coinvariants of $\mathcal{E}_I(X)$ under the action of $\mathcal{A}(X)$.
\end{cor}
\begin{proof}
Recall that $H_{\star}\left(G,M\right)=\mathrm{Tor}_{\star}^{k[G]}\left(k , M\right)$ for a group $G$ and a left $G$-module $M$ \cite[C.3]{Lod}. When $k$ is a field of characteristic zero and $G$ is finite, $M$ is projective \cite[\RNum{1}.8 Exercise 5]{Brown} and so $H_q\left(\mathcal{A}(X), \mathcal{E}_I(X)\right)=0$ for all $q>0$. Therefore the spectral sequence of Theorem \ref{ss-thm} collapses as required.
\end{proof}

\section{An Extension of the Hyperoctahedral Category}
\label{exten-sec}
We extend the hyperoctahedral category $\HC$ by appending an initial object. This extra data yields the structure of a symmetric strict monoidal category under the disjoint union. We can extend the hyperoctahedral bar construction to this new category. One benefit of this new formulation will become apparent when we prove Fiedorowicz's theorem for the hyperoctahedral homology of monoid algebras in Section \ref{Fie-sec}.

\begin{defn}
Let $\DHP$ be the category formed by appending an initial object, the empty set $\emptyset$, to the category $\HC$. For $n\geqslant 0$ we will denote the unique morphism $\emptyset\rightarrow [n]$ by $i_n$.
\end{defn}

\begin{prop}
$\DHP$ is a symmetric strict monoidal category under the disjoint union of sets.
\end{prop}
\begin{proof}
On objects we observe that $[n]\amalg [m] = [n+m+1]$. For $(\phi,g)\in \mr{Hom}_{\DHP}\left([n],[n_1]\right)$ and $(\psi,h)\in \mr{Hom}_{\DHP}\left([m],[m_1]\right)$ the morphism
\[(\phi,g)\amalg (\psi,h)\colon [n+m+1]\rightarrow [n_1+m_1+1]\]
is defined by $(\phi,g)$ acting on the first $n+1$ elements and $(\psi,h)$ acting on the remaining $m+1$ points. One can check that the disjoint union is associative with the empty set as a unit. The transposition functor $[n]\amalg [m]\rightarrow [m]\amalg [n]$ is the identity on objects and on morphisms is defined by precomposition with the block permutation that transposes the first $n+1$ elements with the remaining $m+1$ elements.
\end{proof}

We can extend the hyperoctahedral bar construction to the category $\DHP$.

\begin{defn}
Let $\HBP\colon \DHP\rightarrow \kmod$ be the functor defined as follows. On the subcategory $\HC$ we define $\HBP= \HB$.  We define $\HBP\left(\emptyset\right)=k$. We define $\HBP(i_n)\colon k\rightarrow A^{\otimes n+1}$ to be the inclusion of $k$-algebras.
\end{defn}

\begin{thm}
Let $A$ be an involutive $k$-algebra. There is an isomorphism of graded $k$-modules $HO_{\star}(A)\cong H_{\star}\left(\DHP , \HBP\right)$.
\end{thm}
\begin{proof}
One can easily check that the arguments of \cite[Section 4]{Ault} can be applied in the hyperoctahedral case.
\end{proof}

\section{Fiedorowicz's Theorem for Monoid Algebras}
\label{Fie-sec}
In this section we provide a proof of Fiedorowicz's theorem for the hyperoctahedral homology of monoid algebras, Theorem \ref{monoid-thm}. This result originally appeared in \cite[Proposition 2.3]{Fie} but was never published. It relates the hyperoctahedral homology of a monoid algebra for a monoid with involution to the homology of May's two-sided bar construction. Throughout this section all topological spaces are assumed to be compactly generated weak Hausdorff. We will assume that the category of based topological spaces and the category of based simplicial sets are equipped with the classical Quillen model structures. In particular, a weak homotopy equivalence is a map that induces isomorphisms on homotopy groups. Furthermore, a $C_2$-weak homotopy equivalence is a $C_2$-equivariant map which induces weak homotopy equivalences on the fixed point spaces for both subgroups of $C_2$.

\subsection{Categories and monoids}

Let $C_2\mhyphen\mathbf{Set}_{\star}$ denote the category of finite based sets with a basepoint-preserving $C_2$-action. The morphisms are basepoint-preserving, action-preserving maps of sets. Let $C_2\mhyphen \mathbf{Top}_{\star}$ denote the category of based topological spaces with a basepoint-preserving $C_2$-action. The morphisms are basepoint-preserving continuous maps compatible with the $C_2$-action.

\begin{defn}
\label{total-space}
Let $G$ be a discrete group. Let $\mathbf{E}G$ be the category whose objects are the elements of $G$  with a unique morphism from each object to any other. Let $E_{\star}G$ denote the nerve of this category. Let $EG$ denote the geometric realization of $E_{\star}G$. Let $EG_{+}$ denote the space $EG$ with a disjoint basepoint appended. We will refer to both $E_{\star}G$ and $EG$ as the \emph{total space of $G$}. 
\end{defn}

\begin{defn}
A \emph{monoid with involution} is a monoid $M$, for which the unit element is a non-degenerate basepoint, together with an anti-homomorphism of monoids $M\rightarrow M$, which will be denoted by $m\mapsto \ol{m}$. A \emph{topological monoid with involution} is a monoid with involution equipped with a topology for which the binary operation and the involution are continuous.
\end{defn} 

\begin{rem}
Note that a monoid with involution can be thought of as a topological monoid with involution by equipping it with the discrete topology.
\end{rem}

\subsection{The two-sided bar construction}
\label{bar-constr}
We recall the two-sided bar construction from \cite[Section 9]{MayGILS}.

A \emph{monad} in the category $\mathbf{Top}_{\star}$ is an endofunctor $F$ together with natural transformations $\mu\colon FF\Rightarrow F$ and $\eta\colon \mathbf{1}\Rightarrow F$, where $\mathbf{1}$ is the identity functor on $\mathbf{Top}_{\star}$, satisfying the conditions of \cite[Definition 2.1]{MayGILS}. An \emph{$F$-algebra} is a based topological space $X$ together with a map $\xi\colon FX\rightarrow X$, compatible with $\mu$ and $\eta$ in the sense of \cite[Definition 2.2]{MayGILS}. An \emph{$F$-functor} is a functor $G$, whose source category is $\mathbf{Top}_{\star}$, together with a natural transformation $\lambda\colon GF\Rightarrow G$, compatible with $\mu$ and $\eta$ in the sense of \cite[Definition 9.4]{MayGILS}.

Let $F$ be a monad in $\mathbf{Top}_{\star}$. Let $X$ be an $F$-algebra and let $G$ be an $F$-functor. The \emph{two-sided bar construction} $B_{\star}\left(G,F,X\right)$, is the simplicial space defined in degree $n$ by $B_n\left(G,F,X\right)=GF^nX$, with the $0$-face being induced by the natural transformation $\lambda$, the $n$-face induced by the map $\xi$ and the intermediate faces being induced by the natural transformation $\mu$. The degeneracies are induced by the natural transformation $\eta$.

We denote the based topological space obtained by taking the geometric realization by $B\left(G,F,X\right)$.

\subsection{$A_{\infty}$-operads}
We recall some facts about $A_{\infty}$-operads in $\mathbf{Top}_{\star}$, in particular the trivial operad and the little intervals operad, following \cite[7.4]{FieCS}. Note that we are using the non-symmetric form of $A_{\infty}$-operad, see \cite[3.12, 3.14]{MayGILS}.

An \emph{$A_{\infty}$-operad} $\mathcal{O}$ is a collection of contractible spaces $\llb \mathcal{O}(n): n\geqslant 0\rrb$, such that $\mathcal{O}(0)=\star$, there is a unit element $1\in \mathcal{O}(1)$ and there exist composition maps
\[\gamma\colon \mathcal{O}(n)\times \mathcal{O}(k_1)\times \cdots \times \mathcal{O}(k_n)\rightarrow \mathcal{O}(k_1+\cdots + k_n)\]
satisfying Conditions 1 and 2 of \cite[Definition 1.1]{MayGILS}.

An \emph{$\mathcal{O}$-space} is a based topological space $X$ together with maps $\theta\colon \mathcal{O}(n)\times X^n\rightarrow X$ satisfying Conditions 1 and 2 of \cite[Lemma 1.4]{MayGILS}.

The monad associated to an $A_{\infty}$-operad, $\mathcal{O}$, is the functor $O\colon \mathbf{Top}_{\star}\rightarrow \mathbf{Top}_{\star}$  defined on objects by
\[O(X)= \frac{\coprod_{n\geqslant 0} \mathcal{O}(n)\times X^n}{\approx}\]
where $\approx$ denotes the subspace generated by the equivalence relation induced from the insertion and deletion of basepoints. That is, it is the functor constructed from $\mathcal{O}$ via \cite[Construction 2.4]{MayGILS} where we omit the equivariance conditions.

\begin{eg}
\label{james-eg}
The first example of an $A_{\infty}$-operad is the \emph{trivial operad}, $\mathcal{J}=\llb \star \colon n\geqslant 0\rrb$. A $\mathcal{J}$-space is a topological monoid. The associated monad $J$ is the \emph{James construction}, originally defined in \cite[Section 1]{James}. For $X\in \mathbf{Top}_{\star}$ with basepoint $e$, $J(X)$ is the quotient of $\coprod_{n\geqslant 1} X^n$ by the equivalence relation generated by all identifications of the form
\[\left(x_1,\dotsc ,x_{i-1},e,x_{i+1},\dotsc , x_n\right) \sim \left(x_1,\dotsc x_{i-1},x_{i+1},\dotsc , x_n\right)\]
For a morphism $f\colon X\rightarrow Y$ in $\mathbf{Top}_{\star}$, $J(f)$ is the morphism defined by applying $f$ point-wise.
\end{eg}

\begin{eg}
The second example of an $A_{\infty}$-operad is the \emph{little intervals operad}, $\mathcal{C}_1$, originally due to Boardman and Vogt \cite[Example 5]{BV}. The space
\[\mathcal{C}_1(n)=\llb \left([x_1,y_1],\dotsc , [x_n,y_n]\right) : 0\leqslant x_1<y_1\leqslant x_2 <\cdots \leqslant x_n<y_n\leqslant 1\rrb,\]
where the $[x_i,y_i]$ are closed subintervals of the unit interval $[0,1]$, which is the unit in $\mathcal{C}_1(1)$. The composition maps 
\[\gamma\colon \mathcal{C}_1(n)\times \mathcal{C}_1(k_1)\times \cdots \times \mathcal{C}_1(k_n)\rightarrow \mathcal{C}_1(k_1+\cdots + k_n)\]
are defined by embedding the intervals of $\mathcal{C}_1(k_i)$ into the $i^{th}$ interval of an element of $\mathcal{C}_1(n)$.
The monad $\mathsf{C}_1$ associated to $\mathcal{C}_1$ is called the \emph{little intervals monad}.
\end{eg}

We can extend the James construction, $J$, and the little intervals monad, $\mathsf{C}_1$, to functors on the category $C_2\mhyphen\mathbf{Top}_{\star}$ following \cite{ryb} and \cite[Section 3]{xic} respectively. 

\begin{defn}
\label{equiv-james-defn}
We extend the James construction to a functor $J\colon C_2\mhyphen\mathbf{Top}_{\star}\rightarrow C_2\mhyphen\mathbf{Top}_{\star}$ by defining the $C_2$-action $t\left(x_1,\dotsc , x_n\right) = \left(tx_n,\dotsc , tx_1\right)$.
\end{defn}

\begin{defn}
\label{equiv-c1-defn}
We define a $C_2$-action on $\mathcal{C}_1(n)$, the $n^{th}$ space of the little intervals operad, by reflecting the unit interval about its midpoint. That is,
\[t\left([x_1,y_1],\dotsc ,[x_n,y_n]\right)=\left([1-y_n,1-x_n],\dotsc , [1-y_1, 1-x_1]\right).\]
We extend the little intervals monad to a functor $\mathsf{C}_1\colon C_2\mhyphen\mathbf{Top}_{\star}\rightarrow C_2\mhyphen\mathbf{Top}_{\star}$ by defining the $C_2$-action
\[t\left[\left([x_1,y_1],\dotsc ,[x_n,y_n]\right),\left(x_1,\dotsc ,x_n\right)\right] =\left[t\left([x_1,y_1],\dotsc ,[x_n,y_n]\right), \left(tx_n,\dotsc , tx_1\right)\right].\]
\end{defn}

\begin{prop}
\label{c-1-prop}
There is a natural $C_2$-weak homotopy equivalence $\mathsf{C}_1(X)\rightarrow J(X)$ for any $X\in C_2\mhyphen\mathbf{Top}_{\star}$.
\end{prop}
\begin{proof}
For any $A_{\infty}$-operad, $\mathcal{O}$, there is a unique map of operads $\mathcal{O}\rightarrow \mathcal{J}$ into the trivial operad, which induces a map of monads $O\rightarrow J$ such that $O(X)$ is naturally weakly homotopy equivalent to $J(X)$. One observes that the map of monads $\mathsf{C}_1\rightarrow J$ induced by the unique map of $A_{\infty}$-operads $\mathcal{C}_1\rightarrow \mathcal{J}$, is compatible with the $C_2$-actions of Definitions \ref{equiv-james-defn} and \ref{equiv-c1-defn} and the $C_2$-weak homotopy equivalence follows.
\end{proof}

\subsection{Hyperoctahedral operad}

An \emph{$E_{\infty}$-operad} $\mathcal{O}$ consists of a collection of contractible spaces $\llb \mathcal{O}(n): n\geqslant 0\rrb$, such that there is a free right action of the symmetric group $\Sigma_n$ on $\mathcal{O}(n)$, $\mathcal{O}(0)=\star$, there is a unit element $1\in \mathcal{O}(1)$ and there exist composition maps
\[\gamma\colon \mathcal{O}(n)\times \mathcal{O}(k_1)\times \cdots \times \mathcal{O}(k_n)\rightarrow \mathcal{O}(k_1+\cdots + k_n)\]
satisfying Conditions 1--3 of \cite[Definition 1.1]{MayGILS}.

An \emph{$\mathcal{O}$-space} is a based topological space $X$ together with maps $\theta\colon \mathcal{O}(n)\times X^n\rightarrow X$ satisfying Conditions 1--3 of \cite[Lemma 1.4]{MayGILS}.

The monad associated to an $E_{\infty}$-operad, $\mathcal{O}$, is the functor $O\colon \mathbf{Top}_{\star}\rightarrow \mathbf{Top}_{\star}$  defined on objects by
\[O(X)= \frac{\coprod_{n\geqslant 0} \mathcal{O}(n)\times_{\Sigma_n} X^n}{\approx}\]
where $\approx$ denotes the subspace generated by the equivalence relation induced from the insertion and deletion of basepoints. That is, it is the functor constructed from $\mathcal{O}$ via \cite[Construction 2.4]{MayGILS}.

One example of an $E_{\infty}$-operad is the Barratt-Eccles operad \cite{BE} (see also \cite[1.1]{BeFr}). The $n^{th}$ space of the Barratt-Eccles operad is $E\Sigma_n$, the total space of the symmetric group $\Sigma_n$. We define the analogous $E_{\infty}$-operad where the symmetric groups are replaced with the hyperoctahedral groups.

\begin{defn}
Let $\H$ denote the \emph{hyperoctahedral operad}. For $n\geqslant 1$, let $\H(n)=EH_n$, the total space of $H_n$. Analogously to the Barratt-Eccles operad, the composition operations are determined by the natural maps
\[H_r\times \left(H_{n_1}\times \cdots \times H_{n_r}\right) \rightarrow H_{n_1+\cdots +n_r}\]
where the elements of the $H_{n_i}$ act in blocks and the element of $H_r$ acts on the blocks. The right action of $\Sigma_n$ on $\H(n)$ is the diagonal action.
\end{defn}

\begin{defn}
Let $Z\colon \mathbf{Top}_{\star} \rightarrow \mathbf{Top}_{\star}$ denote the monad associated to the hyperoctahedral operad $\H$.
\end{defn}

\begin{rem} 
\label{e-infty-rem}
Since $\mathcal{H}$ is an $E_{\infty}$-operad, there exists a natural weak homotopy equivalence $Z(X) \simeq \mathsf{C}_{\infty}(X)$ for each $X\in \mathbf{Top}_{\star}$, where $\mathsf{C}_{\infty}$ is the monad associated to the little $\infty$-cubes operad of \cite[Section 4]{MayGILS}.
\end{rem}

\subsection{Hyperoctahedral bar construction for monoids}

\begin{defn}
Let $M$ be a monoid with involution. We define a functor $\HBM\colon \DHP\rightarrow C_2\mhyphen\mathbf{Set}_{\star}$ on objects by
\[\HBM\left([n]\right)=\begin{cases}
M^{n+1} & n\geqslant 0\\
\emptyset & n=-1.
\end{cases}\]
where $M^{n+1}$ denotes the $(n+1)$-fold Cartesian product. $\HBM$ is defined on morphisms in $\HC$ analogously to $\HB$. We define $\HBM\left(i_n\right)$ to be the unique map $\emptyset \rightarrow M^{n+1}$. We call $\HBM$ the \emph{hyperoctahedral bar construction for monoids}.
\end{defn}

\begin{lem}
Let $M$ be a monoid with involution. The hyperoctahedral homology of the monoid algebra $k[M]$ is isomorphic to the homology of the simplicial set $N_{\star}\left(-\downarrow \DHP\right)\times_{\DHP} \HBM$. 
\end{lem}
\begin{proof}
Recall that the free $k$-module functor sends coproducts to direct sums and products to tensor products. One observes that applying the free $k$-module functor to the simplicial set $N_{\star}\left(-\downarrow \DHP\right)\times_{\DHP} \HBM$ we obtain the chain complex $k\left[N_{\star}\left(-\downarrow \DHP\right)\right]\otimes_{\DHP} \mathsf{H}_{k[M]}$ whose homology calculates the hyperoctahedral homology of the monoid algebra $k[M]$.
\end{proof}

\begin{rem}
\label{hocolim-rem}
By considering $M$ equipped with the discrete topology we note that the functor $\HBM$ lands in the category $\C_2\mhyphen\mathbf{Top}_{\star}$ and $N_{\star}\left(-\downarrow \DHP\right)\times_{\DHP} \HBM$ is a simplicial space whose geometric realization is homeomorphic to the homotopy colimit of $\HBM$. That is,
\[HO_{\star}\left(k[M]\right)\cong H_{\star}\left(\hocolim\, \HBM \right).\]
For the rest of this section all homotopy colimits will be taken over the category $\DHP$. In order to ease the typesetting we will henceforth omit the category from the notation.
\end{rem}

\subsection{Technical lemmata}
\label{technical}
Fiedorowicz's theorem, Theorem \ref{monoid-thm}, will relate the hyperoctahedral homology of an involutive monoid algebra of a discrete monoid to the homology of the two-sided bar construction, involving the $C_2$-equivariant little intervals monad $\mathsf{C}_1$ and the monad $\mathsf{C}_{\infty}$ associated to the little $\infty$-cubes operad. The method of proof involves replacing an involutive monoid with a resolution in terms of the bar construction and the $C_2$-equivariant James construction. One applies the hyperoctahedral bar construction and the homotopy colimit and analyses the result. 

In order to prove this theorem we must first analyse the homotopy colimit of the functor $\mathsf{H}_{J(X)}\colon \DHP\rightarrow C_2\mhyphen\mathbf{Top}_{\star}$ for a discrete based $C_2$-space.  The homotopy colimit is the geometric realization of the simplicial space $N_{\star}\left(-\downarrow \DHP\right) \times_{\DHP} \mathsf{H}_{J(X)}$. We will show that this homotopy colimit can be expressed in terms of the monad $Z$ associated to the hyperoctahedral operad.

For ease of indexing we will denote the empty set by $[-1]$ throughout this subsection. For a based topological space $X$, we define $X^0$ to be the basepoint of $X$.

The following lemma shows that we can express the homotopy colimit of the functor $\mathsf{H}_{J(X)}$ in terms of the underlying space $X$.

\begin{lem}
\label{technical-lem-1}
Let $X\in C_2\mhyphen \mathbf{Top}_{\star}$ have the discrete topology. There is an isomorphism of simplicial spaces
\[N_{\star}\left(-\downarrow \DHP\right) \times_{\DHP} \mathsf{H}_{J(X)}\cong \frac{\coprod_{n\geqslant -1} N_{\star}\left([n]\downarrow \DHP\right) \times_{H_{n+1}} X^{n+1}}{\approx} \]
where $\approx$ is the subspace generated by the equivalence relation induced from the insertion and deletion of basepoints.
\end{lem}
\begin{proof}
The technical content of this proof is analogous to \cite[Lemma 33]{Ault}. We will motivate the proof here. By definition, the left hand side is equal to
\[\frac{\coprod_{n\geqslant -1} N_{\star}\left([n]\downarrow \DHP\right) \times J(X)^{n+1}}{\llangle G(\alpha)(x)\otimes y - x\otimes F(\alpha)(y)\rrangle}\]
where $\llangle G(\alpha)(x)\otimes y - x\otimes F(\alpha)(y)\rrangle$ is the subspace generated by the relations on $\DHP$ morphisms, analogously to Definition \ref{tensor-obj}. 

A morphism in $\DHP$ can be written uniquely as a composite $\delta \circ s \circ g$, where $g$ is an automorphism, that is an element of the hyperoctahedral group, $s$ is a surjective order-preserving map and $\delta$ is an injective order-preserving map. 

Taking the quotient of 
\[\coprod_{n\geqslant -1} N_{\star}\left([n]\downarrow \DHP\right) \times J(X)^{n+1}\]
by the subspace generated by the relation for automorphisms we obtain
\[V=\coprod_{n\geqslant -1} N_{\star}\left([n]\downarrow \DHP\right) \times_{H_{n+1}} J(X)^{n+1}.\]
Extending the argument of \cite[Lemma 33]{Ault} to the hyperoctahedral case one can show that any element of the quotient space of $V$ by the subspace generated by the relation for order-preserving surjections can be expressed uniquely as an element of
\[\coprod_{n\geqslant -1} N_{\star}\left([n]\downarrow \DHP\right) \times_{H_{n+1}} X^{n+1}.\]
Finally, one observes that the subspace generated by the relation for injective order-preserving maps is precisely the subspace generated by the relation induced from the insertion and deletion of basepoints.
\end{proof}

\begin{lem}
\label{nerve-equiv}
There exist weak homotopy equivalences of simplicial sets $E_{\star}\Sigma_{n+1} \rightarrow E_{\star} H_{n+1}$ and $E_{\star}H_{n+1}\rightarrow N_{\star}\left([n]\downarrow \DHP\right)$.
\end{lem}
\begin{proof}
All three simplicial sets are connected and contractible. The first weak homotopy equivalence is determined by the inclusion of the symmetric group into the hyperoctahedral group. The second is determined by 
\[\left(g_1g_0^{-1},\dotsc , g_ng_{n-1}^{-1}\right) \mapsto \left(g_0, g_1g_0^{-1},\dotsc , g_ng_{n-1}^{-1}\right)\]
in degree $n$.
\end{proof}

The following lemma demonstrates that the homotopy colimit of the functor $\mathsf{H}_{J(X)}$ is weakly homotopy equivalent to the monad $Z$ of the hyperoctahedral operad applied to the based topological space $EC_{2+}\wedge_{C_2} X$.

\begin{lem}
\label{technical-lem-2}
Let $X\in C_2\mhyphen \mathbf{Top}_{\star}$ have the discrete topology. There is a natural weak homotopy equivalence of simplicial spaces
\[ \frac{\coprod_{n\geqslant -1} N_{\star}\left([n]\downarrow \DHP\right) \times_{H_{n+1}} X^{n+1}}{\approx} \simeq  \frac{\coprod_{n\geqslant -1} E_{\star}H_{n+1} \times_{\Sigma_{n+1}} \left(EC_{2+}\wedge_{C_2}X\right)^{n+1}}{\approx}\]
where $\approx$ is the subspace generated by the equivalence relation induced from the insertion and deletion of basepoints.
\end{lem}
\begin{proof}
Let
\[S=\coprod_{n\geqslant -1} N_{\star}\left([n]\downarrow \DHP\right) \times_{H_{n+1}} X^{n+1}.\]

The total space $EC_{2+}$ is a contractible space with a free $C_2$-action and so the projection map $EC_{2+}\wedge X\rightarrow X$ is a weak homotopy equivalence. That is, we can replace $X$ with a free $C_2$-space. Furthermore, by Lemma \ref{nerve-equiv}, we can replace the nerve of the under-category with the total spaces of the hyperoctahedral groups. We obtain weak homotopy equivalences
\[S\simeq \coprod_{n\geqslant -1} N_{\star}\left([n]\downarrow \DHP\right) \times_{H_{n+1}} \left(EC_{2+}\wedge X\right)^{n+1} \simeq \coprod_{n\geqslant -1} E_{\star}H_{n+1} \times_{H_{n+1}} \left(EC_{2+}\wedge X\right)^{n+1}.\]

We observe that the action of the hyperoctahedral group $H_{n+1}$ on $\left(EC_{2+}\wedge X\right)^{n+1}$ is free. Let $\left(z_0,\dotsc ,z_n; \sigma\right)\in H_{n+1}$. This group element acts firstly by applying the free $\cyc$-action point-wise in the product, according to the elements $z_0,\dotsc ,z_n\in \cyc$. It then freely permutes the factors of the product according to the permutation $\sigma\in \Sigma_{n+1}$. By the properties of free group actions we observe that
\begin{small}
\[S\simeq \coprod_{n\geqslant -1} \frac{\left(EC_{2+}\wedge X\right)^{n+1}}{H_{n+1}} \simeq \coprod_{n\geqslant -1} \frac{\left(EC_{2+}\wedge_{C_2} X\right)^{n+1}}{\Sigma_{n+1}}\simeq \coprod_{n\geqslant -1} E_{\star}\Sigma_{n+1} \times_{\Sigma_{n+1}} \left(EC_{2+}\wedge_{C_2} X\right)^{n+1} .\]
\end{small}

By Lemma \ref{nerve-equiv} we can replace the total spaces of the symmetric groups with the total spaces of the hyperoctahedral groups:
\[S \simeq \coprod_{n\geqslant -1} E_{\star}H_{n+1} \times_{\Sigma_{n+1}} \left(EC_{2+}\wedge_{C_2} X\right)^{n+1}.\]

A routine check shows that these weak homotopy equivalences are compatible with taking the quotient by the subspace generated by the equivalence relation induced from the insertion and deletion of basepoints.
\end{proof}

\begin{cor}
\label{Z-cor}
Let $X\in C_2\mhyphen\mathbf{Top}_{\star}$ have the discrete topology. There are natural weak homotopy equivalences of based topological spaces
\[\mathrm{hocolim}\, \mathsf{H}_{J(X)} \simeq Z\left(EC_{2+}\wedge_{C_2} X\right)\simeq\mathsf{C}_{\infty}\left(EC_{2+}\wedge_{C_2} X\right)\]
where $Z$ is the monad associated to the hyperoctahedral operad and $\mathsf{C}_{\infty}$ is the monad associated to the little $\infty$-cubes operad.
\end{cor}
\begin{proof}
The first weak homotopy equivalence follows from Lemma \ref{technical-lem-2} upon taking geometric realization. The second follows from Remark \ref{e-infty-rem}.
\end{proof}

\begin{lem}
\label{bar-con-lemma}
Let $M$ be a monoid with involution. The weak homotopy equivalence of Corollary \ref{Z-cor} induces a weak homotopy equivalence
\[B_{\star}\left(\mathrm{hocolim}\, \mathsf{H}_{J(-)}, J ,M\right) \rightarrow B_{\star}\left( Z\left(EC_{2+}\wedge_{C_2} - \right), J ,M\right).\]
\end{lem}
\begin{proof}
One observes that the only simplicial map in the bar construction that the equivalence interacts with is the $0$-face. Since the equivalence of Corollary \ref{Z-cor} is natural we have a commuting square
\begin{center}
\begin{tikzcd}
\mathrm{hocolim}\, \mathsf{H}_{J(J(M))}\arrow[r, "\simeq "] \arrow[d, "\p_0 ", swap]& Z\left(EC_{2+}\wedge_{C_2} J(M)\right)\arrow[d, "\p_0 "]\\
\mathrm{hocolim}\, \mathsf{H}_{J(M)}\arrow[r, "\simeq "] & Z\left(EC_{2+}\wedge_{C_2} M\right)
\end{tikzcd}
\end{center}
induced by the map $J(M)\rightarrow M$ as required.
\end{proof}

\subsection{Fiedorowicz's theorem for monoid algebras}

For ease of notation for the remainder of the paper let $\mathsf{E}$ denote the functor $EC_{2+}\wedge_{C_2}-$.

\begin{thm}
\label{monoid-thm}
Let $M$ be a discrete monoid with involution. There is an isomorphism of graded $k$-modules $HO_{\star}\left(k[M]\right)\cong H_{\star}\left(B\left(\mathsf{C}_{\infty}\mathsf{E}, \mathsf{C}_1 ,M\right)\right)$.
\end{thm}
\begin{proof}
Recall the $C_2$-equivariant James construction $J$ from Definition \ref{equiv-james-defn}. By \cite[Lemma 2.5]{Dunn} there is a $C_2$-equivariant weak homotopy equivalence $B\left(J , J , M\right)\rightarrow M$, where $B\left(J , J ,M\right)$ is the geometric realization of the based simplicial $C_2$-space $B_{\star}\left(J ,J ,M\right)$ given by the two-sided bar construction. We therefore have isomorphisms
\[HO_{\star}\left(k[M]\right) \cong H_{\star}\left(\mathrm{hocolim} \, \HBM\right) \cong H_{\star}\left(\mathrm{hocolim}\, \mathsf{H}_{B\left(J,J,M\right)}\right).\]
By \cite[Lemma 9.7]{MayGILS} we can take the hyperoctahedral bar construction and the homotopy colimit inside the bar construction and apply Corollary \ref{Z-cor} to obtain, 
\[\mathrm{hocolim}\, \mathsf{H}_{ B\left(J,J,M\right)} \cong B\left(\mathrm{hocolim}\, \mathsf{H}_{J(-)},J,M\right)\simeq B\left(Z\mathsf{E},J,M\right).\]
Finally, there is a weak homotopy equivalence $B\left(Z\mathsf{E},J,M\right) \simeq B\left(\mathsf{C}_{\infty}\mathsf{E},\mathsf{C}_1,M\right)$ by Corollary \ref{Z-cor} and Proposition \ref{c-1-prop}.
\end{proof}

\section{Hyperoctahedral Homology of Group Algebras}
\label{grp-alg-sec}
We begin this section by demonstrating that hyperoctahedral homology of a group algebra, where the involution is induced from sending a group element to its inverse, is the homology of an infinite loop space on the two-sided bar construction. We will then show that in the case where this involution is a free $C_2$-action away from the basepoint one can say more; the hyperoctahedral homology in this case is the homology of the $C_2$-fixed points under the involution of an equivariant infinite loop space. In particular, this holds for all discrete groups of odd order.

Let $Q$ denote the free infinite loop space functor $\Omega^{\infty}\Sigma^{\infty}\colon \mathbf{Top}_{\star}\rightarrow \mathbf{Top}_{\star}$.

\begin{prop}
Let $G$ be a discrete group. Let the group algebra $k[G]$ have involution determined by sending a group element to its inverse. There is an isomorphism of graded $k$-modules $HO_{\star}\left(k[G]\right) \cong H_{\star}\left(QB\left(\mathsf{E},\mathsf{C}_1,G\right)\right)$.
\end{prop}
\begin{proof}
Since $G$ is a group, $B\left(\mathsf{C}_{\infty}\mathsf{E}, \mathsf{C}_1 ,G\right)$ is group-like. We can therefore apply the version of the recognition principle given in \cite[\RNum{7}, Theorem 3.1(ii)]{May77} to deduce a weak homotopy equivalence $B\left(\mathsf{C}_{\infty}\mathsf{E}, \mathsf{C}_1 ,G\right) \simeq B\left(Q\mathsf{E}, \mathsf{C}_1 ,G\right)$. By \cite[Lemma 9.7]{MayGILS} there is an isomorphism $B\left(Q\mathsf{E}, \mathsf{C}_1 ,G\right)\cong QB\left(\mathsf{E}, \mathsf{C}_1 ,G\right)$. The proposition then follows from Theorem \ref{monoid-thm}.
\end{proof}

Let $\Phi\colon C_2\mhyphen\mathbf{Top}_{\star}\rightarrow \mathbf{Top}_{\star}$ denote the $C_2$-fixed point functor. 

Let $\Sigma_{C_2}\colon C_2\mhyphen\mathbf{Top}_{\star}\rightarrow C_2\mhyphen\mathbf{Top}_{\star}$ denote the $C_2$-equivariant suspension functor. For $X\in C_2\mhyphen\mathbf{Top}_{\star}$, the space $\Sigma_{C_2}(X)=S^1\wedge X$ has involution defined by $\ol{\left[\left(t,x\right)\right]}=\left[\left(1-t , \ol{x}\right)\right]$.

Let $\Omega_{C_2}\colon C_2\mhyphen\mathbf{Top}_{\star}\rightarrow C_2\mhyphen\mathbf{Top}_{\star}$ denote the $C_2$-equivariant loop functor. For $X\in C_2\mhyphen\mathbf{Top}_{\star}$, the space $\Omega_{C_2}(X)= \mr{Hom}_{\mathbf{Top}_{\star}}\left(S^1,X\right)$ has involution defined by $\ol{f}(t)=\ol{f(1-t)}$. 
 
Let $Q_{C_2}\colon C_2\mhyphen\mathbf{Top}_{\star}\rightarrow C_2\mhyphen\mathbf{Top}_{\star}$ denote the $C_2$-equivariant free infinite loop space functor. That is, $Q_{C_2}=\Omega_{C_2}^{\infty}\Sigma_{C_2}^{\infty}$.

\begin{rem}
\label{equiv-Q-rem}
Similarly to the non-equivariant case there is a natural $C_2$-weak homotopy equivalence $ \Omega_{C_2}Q_{C_2}\Sigma_{C_2}(X) \simeq Q_{C_2}(X)$.
\end{rem}

\begin{prop}
\label{fixed-pt-prop}
Let $X\in C_2\mhyphen \mathbf{Top}_{\star}$ have a free $C_2$-action away from the basepoint. There is a natural weak homotopy equivalence of topological spaces $\Phi\left(Q_{C_2}(X)\right) \simeq Q\mathsf{E}(X)$.
\end{prop}
\begin{proof}
A theorem of tom Dieck \cite[Corollary 2.3]{carlsson} tells us that for a based $C_2$-space $X$ there is a weak homotopy equivalence $\Phi\left(Q_{C_2}(X)\right)\simeq Q\mathsf{E}(X) \times Q\left(\Phi(X)\right)$.
Since we are assuming that the $C_2$-action on $X$ is free, the fixed points $\Phi(X)$ are trivial and so $Q\left(\Phi(X)\right)$ is trivial. We therefore deduce the required weak homotopy equivalence.
\end{proof}

In the non-equivariant case, Fiedorowicz \cite[Corollary 7.9]{FieCS} showed that for any topological monoid $M$ there is a weak homotopy equivalence $B\left(\Sigma, \mathsf{C}_1,M\right) \rightarrow BM$, between the bar construction and the classifying space of $M$. We will show that this extends to a $C_2$-equivariant weak homotopy equivalence when $M$ is a topological monoid with involution.

\begin{defn}
\label{bar-constr-inv}
For a topological monoid with involution $M$, we define an involution on $B_n\left(\Sigma_{C_2},\mathsf{C}_1, M\right)=S^1\wedge \mathsf{C}_1^nM$, the $n^{th}$ level of the bar construction, by $\ol{\left[\left(t,x\right)\right]}=\left[\left(1-t , \ol{x}\right)\right]$ where $t\in S^1$ and $x\in \mathsf{C}_1^nM$. We define the geometric realization $B\left(\Sigma_{C_2},\mathsf{C}_1, M\right)$ to have the involution induced by the level-wise involution.
\end{defn}

\begin{defn}
\label{class-sp-inv}
Let $M$ be a topological monoid with involution. We define an involution on the classifying space $BM$ by $ \ol{\left[\left(t_0,\dotsc , t_n\right),\left(m_1,\dotsc , m_n\right)\right]}=\left[\left(t_n,\dotsc , t_0\right),\left(\ol{m_n},\dotsc ,\ol{m_1}\right)\right]$.
\end{defn}

\begin{rem}
Let $G$ be a discrete group with involution given by sending an element to its inverse. The involution on $G$ induces an involution on the classifying space $BG$ which is equivariantly homotopic to the involution of Definition \ref{class-sp-inv} by \cite[Construction 3.1]{BF}.
\end{rem}

\begin{prop}
\label{equiv-Q-prop}
Let $M$ be a topological monoid with involution. There exists a $C_2$-weak homotopy equivalence $B\left(Q_{C_2},\mathsf{C}_1,M\right)\rightarrow \Omega_{C_2}Q_{C_2} BM$.
\end{prop}
\begin{proof}
A routine check shows that the chain of equivalences constructed in \cite[Theorems 7.3 and 7.8]{FieCS} are equivariant with respect to the actions defined in \cite[Section 1]{ryb}, yielding a weak $C_2$-equivariant weak homotopy equivalence $B\left(\Sigma_{C_2}, \mathsf{C}_1 , M\right)\rightarrow BM$, which sends the involution of Definition \ref{bar-constr-inv} to the involution of Definition \ref{class-sp-inv}.
Apply the functor $\Omega_{C_2}Q_{C_2}$ to both sides of the equivalence. By \cite[Lemma 9.7]{MayGILS} we can take these functors inside the bar construction on the left hand side. The result then follows from Remark \ref{equiv-Q-rem}.
\end{proof}

\begin{thm}
\label{main-theorem}
Let $G$ be a discrete group of odd order. Let the group algebra $k[G]$ and the classifying space $BG$ have the involution induced by sending a group element to its inverse. There is an isomorphism of graded $k$-modules $HO_{\star}\left(k[G]\right)\cong H_{\star}\left( \Phi \Omega_{C_2}Q_{C_2} BG\right)$.
\end{thm}
\begin{proof}
By Theorem \ref{monoid-thm} we have an isomorphism $HO_{\star}\left(k[G]\right) \cong H_{\star}\left(B\left(\mathsf{C}_{\infty}\mathsf{E}, \mathsf{C}_1 ,G\right)\right)$. Since $G$ is a group, $B\left(\mathsf{C}_{\infty}\mathsf{E}, \mathsf{C}_1 ,G\right)$ is group-like and so by \cite[\RNum{7}, Theorem 3.1(ii)]{May77} there is a weak homotopy equivalence $B\left(\mathsf{C}_{\infty}\mathsf{E}, \mathsf{C}_1 ,G\right) \simeq B\left(Q\mathsf{E}, \mathsf{C}_1 ,G\right)$. Since $G$ has odd order, the involution is free away from the basepoint and we can apply Proposition \ref{fixed-pt-prop}, followed by \cite[Lemma 9.7]{MayGILS} and Proposition \ref{equiv-Q-prop} to obtain the following chain of equivalences
\[B\left(Q\mathsf{E}, \mathsf{C}_1 ,G\right) \simeq B\left(\Phi Q_{C_2},\mathsf{C}_1, G\right)\cong \Phi B\left( Q_{C_2},\mathsf{C}_1, G\right)\simeq \Phi \Omega_{C_2}Q_{C_2} BG\]
from which the theorem follows.
\end{proof}

\bibliographystyle{alpha}
\bibliography{HO-refs}

\end{document}